\documentclass[12pt,leqno]{amsart}
\usepackage{amsmath,amssymb,amsfonts}
\usepackage[all]{xy}

\usepackage[T5,T1]{fontenc}
\usepackage{mathpazo}
\usepackage{hyperref}
\usepackage{a4wide}

\theoremstyle{plain}
\newtheorem{thm}{Theorem}[section]
\newtheorem{lem}[thm]{Lemma}
\newtheorem{cor}[thm]{Corollary}
\newtheorem{prop}[thm]{Proposition}

\theoremstyle{definition}
\newtheorem{defn}[thm]{Definition}
\newtheorem{ex}[thm]{Example}
\newtheorem{examples}[thm]{Examples}

\newtheorem{question}[thm]{Question}

\newtheorem{rmk}[thm]{Remark}
\newtheorem{rmks}[thm]{Remarks}

\newcommand{\surj}{\twoheadrightarrow}

\newcommand{\rank}{{\rm rank}}

\newcommand{\F}{{\mathbb F}}

\newcommand{\N}{{\mathbb N}}

\newcommand{\Q}{{\mathbb Q}}

\newcommand{\U}{{\mathbb U}}

\newcommand{\Z}{{\mathbb Z}}

\def\NDT{{\fontencoding{T5}\selectfont Nguy\~ \ecircumflex n Duy T\^an}}
 
\begin{document}
\title[Zassenhaus filtration]{Dimensions of Zassenhaus filtration subquotients of some pro-$p$-groups}

 \author{ J\'an Min\'a\v{c}, Michael Rogelstad and \NDT}
\address{Department of Mathematics, Western University, London, Ontario, Canada N6A 5B7}
\email{minac@uwo.ca}
\address{Department of Mathematics, Western University, London, Ontario, Canada N6A 5B7}
\email{mrogelst@uwo.ca}
 \address{Department of Mathematics, Western University, London, Ontario, Canada N6A 5B7 and Institute of Mathematics, Vietnam Academy of Science and Technology, 18 Hoang Quoc Viet, 10307, Hanoi - Vietnam } 
\email{dnguy25@uwo.ca}
\thanks{JM is partially supported  by the Natural Sciences and Engineering Research Council of Canada (NSERC) grant R0370A01. MR is partially supported by a CGS-D scholarship.
NDT is partially supported  by the National Foundation for Science and Technology Development (NAFOSTED)}
 \begin{abstract}
 We compute the $\F_p$-dimension of an $n$-th graded piece $G_{(n)}/G_{(n+1)}$ of the Zassenhaus filtration    for various finitely generated pro-$p$-groups $G$. These groups include finitely generated free pro-$p$-groups, Demushkin pro-$p$-groups and their free pro-$p$ products. We provide a unifying principle for deriving these dimensions.
\end{abstract}
\maketitle
\section{Introduction}
Recall that for a profinite group $G$ and a prime number $p$,  the Zassenhaus ($p$-)filtration $(G_{(n)})$ of $G$ is defined inductively by
\[
G_{(1)}=G, \quad G_{(n)}=G_{(\lceil n/p\rceil)}^p\prod_{i+j=n}[G_{(i)},G_{(j)}],
\]
where $\lceil n/p \rceil$ is the least integer which is greater than or equal to $n/p$. (Here for two closed subgroups $H$ and $K$ of $G$,  $[H,K]$ means the smallest closed subgroup of $G$ containing the  commutators $[x,y]=x^{-1}y^{-1}xy$, $x\in H, y\in K$. Similarly, $H^p$ means the smallest closed subgroup of $G$ containing  the $p$-th powers $x^p$, $x\in H$.) Zassenhaus filtrations of groups were introduced in \cite{Zas} and are now recognized as being of fundamental importance in determining the structure and properties of various types of groups. For example, in the case of absolute Galois groups, these filtrations and their subquotients have recently been investigated in \cite{CEM, Ef1,Ef2, EM,MT, EMT}.  In the case of arbitrary groups, this filtration has also been referred to as the dimension series, with the subgroups $G_{(n)}$ being called  the dimension subgroups in characteristic $p$ (see \cite[Chapters 11, 12]{DDMS}).  Our goal is to develop a method for determining the $\F_p$-dimension of subquotients of the Zassenhaus filtration in the case of finitely generated pro-$p$ groups. 

Let $G$ be a finitely generated pro-$p$-group. For each $n\geq  1$, we set
\[
 c_n(G) =\dim_{\F_p}(G_{(n)}/G_{(n+1)}).
\]
Note that since $G$ is finitely generated, $c_n(G)$ is finite for every $n\geq 1$  (see Section 2 for more details).
We will proceed to derive an explicit formula for $c_n(G)$ for various families of groups $G$, including finitely generated free pro-$p$-groups,  Demushkin groups, and free pro-2 products of finitely many copies of the cyclic  group $C_2$ of order 2. Galois theory provides much of the underlying motivation, as many of these groups are  realizable as Galois groups of maximal $p$-extensions of local fields (see \cite{De1}, \cite{Sha}) and other fields (see \cite[Proposition 1.3]{EH}).  
Shafarevich \cite{Sha} demonstrated that for certain fields $F$ not containing primitive $p$-th roots of unity, one could show that the Galois group of the maximal $p$-extension of $F$ was a free pro-$p$-group simply by determining the cardinality of some of its filtration quotients. 

In Remarks~\ref{rmks:characterization} (1) we show that the numbers $c_n(G)$, $n=1,2,\ldots$, are sufficient to determine  finitely generated free pro-$p$-groups in the family of all finitely generated pro-$p$-groups. 
 In Remarks~\ref{rmks:characterization} (2) we are able to determine finitely generated free pro-p groups in the family of all Galois groups of the maximal $p$-extensions of fields containing a primitive $p$-th root of unity by just  two numbers, $c_1(G)$ and $c_2(G)$.
In Remarks~\ref{rmks:SAP} we show that $c_1(G)$ and $c_2(G)$ are sufficient to determine the Galois groups of the maximal $2$-extensions of Pythagorean fields in two significant cases. 
 In Subsection 4.2 we study groups $G$ which are the free products of several copies of the cyclic group of order 2 in the category of pro-$2$-groups. These groups  can be realized as the Galois groups of the maximal 2-extensions of Pythagorean SAP fields, and therefore they are significant in Galois theory. Each such group G contains a free pro-2-subgroup $H$ of index 2. In Corollary~\ref{cor:quotient} we are able to use knowledge of the numbers $c_n(G)$ and $c_n(H)$,
to obtain the interesting relation $H_{(n)} = H\cap G_{(n)}$ for each $n\geq 2$. This is yet another example illustrating that the numbers $c_n(G)$ can be very useful in group theory and Galois theory.

In this paper we provide a unifying principle for deriving the dimensions $c_n(G)$ in a number of interesting cases. We observe that the formulas obtained look  simple, elegant, and potentially useful. 
We would also like to note that when $S$ is a finitely generated free pro-$p$ group, a formula for $c_n(S)$ is implicitly given in  \cite{Ga}, where an $\F_p$-basis for $S_{(n)}/S_{(n+1)}$ is provided.

When we interpret the groups $G$ considered in this paper as Galois groups, our formulas  lead to formulas for the order of related Galois groups. For example, if $G$ is isomorphic to the maximal pro-$p$-quotient $G_F(p)$ of the absolute Galois group $G_F$ of a field $F$, and if we denote by $F_{(n)}$ the fixed field of $G_F(p)_{(n)}$, then $|{\rm Gal}(F_{(n)}/F)|=p^{\sum_{i=1}^{n-1} c_i(G)}$. As indicated above, these Galois groups play a fundamental role in current Galois theory. Furthermore, we observe in Sections 3 and 5 that our formulas also lead to the determination of the minimal number of topological generators of $G_{(n)}$, for $G$ a free pro-$p$-group or a Demushkin pro-$p$-group.
 In fact the orders of Galois groups ${\rm Gal}(F_{(n)}/F)$ are of considerable interest in  current Galois theory research. In particular, in \cite{Ef1,MT, EMT}, based partially on the special cases in \cite{EM,MS2}, the Kernel Unipotent Conjecture was formulated. If this conjecture is true, we would obtain a characterization of $G_F(p)_{(n)}$, where $n\geq 3,$ as the intersection of the kernels of all Galois representations $\rho\colon G_F(p)\to \U_n(\F_p)$. In order to prove the Kernel Unipotent Conjecture in the case when $G_F(p)$ is finitely generated, one may try to produce enough such representations. However, in order to check whether the intersection of the kernels of given representations is in fact $G_F(p)_{(n)}$, it would be useful to know  $|{\rm Gal}(F_{(n)}/F)|$. This strategy resembles  the previous successful strategy of Shafarevich,  mentioned above. Another very interesting project in current Galois theory is to study the possible Koszul duality relating the Galois cohomology algebra $H^*(G,\F_p)$ to the Lie algebra $\bigoplus_{n=1}^\infty G_{(n)}/G_{(n+1)}$ and its universal enveloping algebra. In order to check some corollaries of this possible Koszul duality, determination of the numbers $c_n(G)$ could play an important role.

 The structure of our paper is as follows: In Section 2 we discuss Hilbert-Poincar\'e series and provide a general formula for $c_n(G)$ (see Theorem~\ref{thm:general}). In Section 3 we provide an explicit formula for $c_n(G)$ when $G$ is a free pro-$p$-group of finite rank.  In Section 4 we treat the case when $G$ is a free pro-2 product of finite copies of $C_2$. 
 We also show that in some significant special cases, knowledge of just $c_1(G)$ and $c_2(G)$ is sufficient to determine certain Galois groups within  large families of 
 pro-$2$-groups. 
 In Section 5 we treat the case in which $G$ is a Demushkin group.  In the last section we discuss some other groups.
 \\
 \\
 {\bf Acknowledgements: } 
 The first author gratefully acknowledges discussions with I. Efrat, J. Labute and A. Topaz; the latter two having provided some extra motivation for the work in this paper. All of the authors would  like to thank the referee for  valuable suggestions related to the exposition.
\section{Hilbert-Poincar\'e series}
Let $F$ be a unital commutative ring. A graded free $F$-module $V=\bigoplus_{i=0}^\infty V_n$ is called {\it locally finite} if ${\rm rank}_F(V_n)<\infty$ for all $n\geq 0$. For such a graded free $F$-module $V$, the {\it Hilbert-Poincar\'e series} $P_V(t)\in \Z[[t]]$ of $V$ is the formal power series
\[
P_V(t)=\sum_{n=0}^\infty {\rm rank}_F(V_n) t^n.
\]

Let $G$ be a finitely generated pro-$p$-group. It is convenient to use the completed group algebra $\F_p[[G]]$ of $G$ over $\F_p$
\[
\F_p[[G]]:= \varprojlim_{N} \F_p[G/N].
\]
Thus $\F_p[[G]]$ is the topological inverse limit of the usual group rings $\F_p[G/N]$, where $N$ runs through open normal subgroups of $G$. A standard reference for completed group rings is \cite[Chapter 5]{NSW}. We also use the convenient references \cite[Chapter 7]{Ko} and \cite[Chapters 7 and 12]{DDMS}. We recall that $I(G)\subset \F_p[[G]]$  denotes the augmentation ideal of $\F_p[[G]]$ which is the closed two-sided ideal of $\F_p[[G]]$ generated by the elements $g-1$, for $g\in G$. Thus if $\epsilon\colon \F_p[[G]]\to \F_p$ is the homomorphism such that $\epsilon(g)=1$ for all $g\in G$, then $I(G)=\ker\epsilon$.
We denote by $I^n(G)$ the $n$-th power of the augmentation ideal $I(G)$. 
There are two graded $\F_p$-algebras associated to $G$ and $\F_p[[G]]$ respectively, which are defined by
\[
{\rm gr}(G)= \bigoplus_{n\geq 1}  G_{(n)}/G_{(n+1)} \quad  \text{ and } \quad {\rm gr}(\F_p[[G]]) =\bigoplus_{n\geq 0} I^n(G)/I^{n+1}(G).
\]
Then ${\rm gr}(G)$ is a restricted Lie algebra. (See \cite[Chapter 12]{DDMS}.) Furthermore since $G$ is finitely generated, the two graded algebras ${\rm gr}(G)$ and ${\rm gr}(\F_p[[G]])$ are locally finite (see \cite[Section 7.4]{Ko}). We recall that $c_n(G)=\dim_{\F_p} G_{(n)}/G_{(n+1)}$ and we let  $a_n(G):= \dim_{\F_p} I^n(G)/I^{n+1}(G)$. 
As  pointed out in \cite[page 312]{DDMS}, $a_n(G)=\dim_{\F_p} I_0^n(G)/I_0^{n+1}(G)$, where $I_0(G)$ is the augmentation ideal  of $\F_p[G]$ - the usual group ring of $G$. Thus our results below apply to this case as well. In several places we use results from discrete groups which extend in a straightforward way to pro-$p$-groups. We usually mention this, but in some cases we omit explicitly mentioning  such a standard extension.
The following theorem, Theorem~\ref{thm:JL}, is a consequence of a beautiful theory of Jennings and Lazard \cite[Chapters 11 and 12]{DDMS} viewing the Zassenhaus filtration subgroups $G_{(n)}$ as dimension subgroups. (See also \cite{Qu}.)
\begin{thm}[Jennings-Lazard]
\label{thm:JL}
 Let the notation be as above.
\begin{enumerate}
\item The graded algebra ${\rm gr}(\F_p[[G]])$ is a restricted universal enveloping algebra of ${\rm gr}(G)$.
\item We have
 \begin{equation}
  \label{eq:fundamental}  P_{{\rm gr}(\F_p[[G]])}(t) = 
 \sum_{n=0}^\infty a_n(G) t^n=\prod_{n=1}^\infty \left(\frac{1-t^{np}}{1-t^n}\right)^{c_n(G)}.
\end{equation}
\end{enumerate}
\end{thm}
\begin{proof} (1) See \cite[Theorem 12.8]{DDMS}. 

 (2) See \cite[Theorem 12.16]{DDMS} (see also \cite[Proposition 2.3]{Er}). 
\end{proof}
The following lemma is an important technical tool which allows us to derive our results in a concise  way. It relies on one fundamental result of Lichtman and also on a simple, but quite remarkable formula which can be traced back to the work of Lemaire in \cite[Chapter 5]{Le}. 
\begin{lem} 
\label{lem:free product}
Let $G_1$ and $G_2$ be two finitely generated pro-$p$-groups. Let $G=G_1*G_2$ be the free product of $G_1$ and $G_2$ in the category of  pro-$p$-groups. Then
\[
P_{{\rm gr}(\F_p[[G]])}(t) = (P_{{\rm gr}(\F_p[[G_1]])}^{-1}(t)+P_{{\rm gr}(\F_p[[G_2]])}^{-1}(t)-1)^{-1}.
\]
\end{lem}

\begin{proof} By \cite[Theorem 1]{Li}, the graded $\F_p$-algebra ${\rm gr}(\F_p[[G]])$ is a free product (i.e., a categorical coproduct) 
of ${\rm gr}(\F_p[[G_1]])$ and ${\rm gr}(\F_p[[G_2]])$. The statement then follows from \cite[Equation (1.2), page 56]{PP}.
\end{proof}
\begin{rmk}
\label{rmk:direct product}
 Let $G=G_1\times G_2$ be the direct product of two finitely generated pro-$p$-groups $G_1$ and $G_2$. We first observe that every commutator in $G$ is the product of a commutator in $G_1$ and a commutator in $G_2$, and every $p$-power in  $G$ is the product of a $p$-power in $G_1$ and a $p$-power in $G_2$. Then we can check that $G_{(n)}=(G_1)_{(n)}\times (G_2)_{(n)}$, and that 
\[\frac{G_{(n)}}{G_{(n+1)}}\simeq\frac{(G_1)_{(n)}}{(G_1)_{(n+1)}}\times \frac{(G_2)_{(n)}}{(G_2)_{(n+1)}}.\]
This implies that $c_{n}(G)=c_{n}(G_1)+c_{n}(G_2)$, and  that
\[
P_{{\rm gr}(\F_p[[G]])}(t) = P_{{\rm gr}(\F_p[[G_1]])}(t)\cdot P_{{\rm gr}(\F_p[[G_2]])}(t).
\]
In fact, since ${\rm gr}(G)\simeq {\rm gr}(G_1) \oplus {\rm gr}(G_2)$, one can show that 
\[{\rm gr}(\F_p[[G]])\simeq {\rm gr}(\F_p[[G_1]]) \otimes {\rm gr}(\F_p[[G_2]]).\]
\end{rmk}
\begin{examples}
\mbox{}
In our examples below, $d$ can be any natural number, and in (3), $d=0$ is also meaningful.
\begin{enumerate} 
\item If $G$ is a free pro-$p$-group of rank $d$, then (see Section 3)
\[ P_{{\rm gr}(\F_p[[G]])}(t)= \frac{1}{1-dt}.
\]

\item If $G=C_p$ is a cyclic group of order $p$, then
\[ P_{{\rm gr}(\F_p[[G]])}(t)= 1+t+\cdots+t^{p-1}.
\]
Indeed, since ${\rm gr}(C_p)=C_p$,  the graded algebra ${\rm gr}(\F_p[[G]])$, which is a universal enveloping algebra of ${\rm gr}(C_p)$ by Theorem~\ref{thm:JL}, is isomorphic to $\F_p[X]/(X^p)$,  the preceding statement  follows.
\\
\item If $G=C_p*\cdots*C_p$ is a free product of $d+1$ copies of $C_p$ the cyclic group of order $p$, then 
\[ P_{{\rm gr}(\F_p[[G]])}(t)= \frac{1+t+\cdots + t^{p-1},}{1-dt-\cdots - dt^{p-1}}.
\]
This follows by induction on $d$, and by using (2) above, and Lemma~\ref{lem:free product}. 
\\
\item If $G=\Z_p^d$, then 
\[
P_{{\rm gr}(\F_p[[G]])}(t)=\frac{1}{(1-t)^d}. 
\]
This follows from Remark~\ref{rmk:direct product} and (1) above.
\\
\item If $G:=\Z_2^d \rtimes C_2$, where  $C_2=\langle x \rangle$ and  the action of $C_2$ on $\Z_2^d$ is given by $xyx=y^{-1}$, for all $y\in \Z_2^d$, then (see Corollary \ref{cor:superPy})
\[
 P_{{\rm gr}(\F_2[[G]])}(t)=\frac{1+t}{(1-t)^d}\prod_{i=1}^{\infty}\frac{1}{1-t^{2i+1}}. 
 \]

\item If $G$ is a Demushkin pro-$p$-group of rank $d$, then (see Section 5)
\[ P_{{\rm gr}(\F_p[[G]])}(t)= \frac{1}{1-dt+t^2}.
\]
\item If $G$ is a free product of a cyclic group of order 2 and a free pro-$2$-group of rank $d$, then 
\[
 P_{{\rm gr}(\F_p[[G]])}(t)= \frac{1+t}{1-dt-dt^2}.
\]
This follows by using Lemma~\ref{lem:free product} and (1)-(2) above.
\qed
\end{enumerate}
\end{examples}

Below we shall describe a general method for deriving a formula for $c_n(G)$ if we know the Hilbert-Poincar\'e series  $P_{{\rm gr}(\F_p[[G]])}(t)$. 
So we assume that we are given a power series $P(t)=1+\sum_{n\geq 1}a_nt^n\in \Z[[t]]$. We define $c_n,n=1,2,\ldots$ by
\[
 P (t)=1+\sum_{n\geq 1}a_n t^n=\prod_{n=1}^\infty \left(\frac{1-t^{np}}{1-t^n}\right)^{c_n}.
\]
We write $\log P(t)= \sum_{n\geq 1} b_n t^n$. 
 We shall derive a formula for $c_n$ using the values $b_1,\ldots,b_n$. To do this, it is convenient to introduce a  new auxiliary sequence $w_1,w_2,\ldots$ defined below.

Taking logarithms and using $\log(\dfrac{1}{1-t})=\sum \limits_{\nu=1}^\infty \dfrac{1}{\nu}t^\nu$, we obtain
\[
 \sum_{n=1}^\infty b_n t^n =\sum_{m=1}^\infty c_m \sum_{\nu=1}^\infty \frac{1}{\nu} (t^{m\nu}-t^{mp\nu}).
\]
Equating the coefficients of $t^n$, we obtain
\[
b_n=\sum_{m\nu=n} \frac{1}{\nu}c_m - \sum_{mp\nu=n}\frac{1}{\nu}c_m.
\]
Hence
\[
 nb_n= \sum_{m\mid n} m c_m -\sum_{mp\mid n} mp c_m.
\]

Recall that for two integers $a$ and $b$, the symbol $a\mid b$ means that $a$ divides $b$. Now we define the sequence $w_n, n=1,2,\ldots$ by
\[
w_n=\frac{1}{n}\sum_{m\mid n} \mu(n/m) mb_m.
\]
Then by the M{\"o}bius inversion formula,
\[
nb_n= \sum_{m\mid n} m w_m.
\]
Here $\mu$ is the M\"obius function: for a positive integer $d$, 
\[
 \mu(d)=
\begin{cases}
(-1)^r & \text{ if $d$ is a product of $r$ distinct prime numbers},\\
0 &\text{ otherwise}.
\end{cases}
\]
\begin{rmk}
\label{rmk:wn}
 From the definition of $w_n$ we see that
 \[
 P (t)=1+\sum_{n\geq 1}a_n t^n=\prod_{n=1}^\infty \frac{1}{(1-t^n)^{w_n}}.
\]
\end{rmk}

\begin{lem}
\label{lem:coprime}
 If $(n,p)=1$ then $c_n=w_n$. 
\end{lem}
\begin{proof}
 Assume that $(n,p)=1$. Then we have
\[nb_n=\sum_{m\mid n} mc_m.\]
Hence by the M{\"o}bius inversion formula, we have
\[
 c_n=\frac{1}{n}\sum_{m\mid n} \mu(n/m) mb_m=w_n.
\qedhere
\]
\end{proof}

\begin{lem}
\label{lem:not coprime}
 If $p$ divides $n$, then we have
\[
c_n = c_{n/p}+w_n.
\]
\end{lem}
\begin{proof}
We proceed by induction on $n$. Clearly $c_p-c_1=\dfrac{pb_p-b_1}{p}=w_p$, hence the statement is true for $n=p$. Therefore we assume now that $n>p$ and $p\mid n$. We assume that the statement is true for every $m$ such that $p\mid m\mid n$, $m\not=n$. We are going to show that the statement is also true for $n$.

 We have
\[
\begin{aligned}
 nb_n &=\sum_{m\mid n} m c_m -\sum_{pm\mid n} pm c_m\\
&=\sum_{m\mid n} m c_m -\sum_{p\mid m\mid n} m c_{m/p}\\
&= \sum_{m\mid n, (m,p)=1} m c_m +\sum_{p\mid m\mid n} m (c_m-c_{m/p})\\
&= \sum_{m\mid n, (m,p)=1} m w_m +\sum_{p\mid m\mid n,m\not= n} m w_m + n(c_n-c_{n/p})\\
&=\sum_{m\mid n, m\not= n} m w_m +n(c_n-c_{n/p}). 
\end{aligned}
\]
Combining with 
\[
 nb_n=\sum_{m\mid n} mw_m,
\]
we obtain $c_n-c_{n/p}=w_n$.
\end{proof}

\begin{prop}
\label{prop:key}
 If $n=p^k m $ with $(m,p)=1$, then 
\[
c_n =w_m +w_{pm}+\cdots + w_{p^km}.
\]
\end{prop}
\begin{proof}
This follows from the  above two lemmas.
\end{proof}  
\begin{thm}
\label{thm:general}
  Let $G$ be a finitely generated pro-$p$-group. We write
   \[ \log P_{{\rm gr}(\F_p[[G]])}(t)=\sum_{n\geq  1}b_nt^n\in \Q[[t]],\] and we define $w_n(G)$ by
\[
w_n(G):=\frac{1}{n}\sum_{m\mid n} \mu(n/m) mb_m.
\]
Let  $n=p^k m $ with $(m,p)=1$. Then 
\[
c_n(G) =w_m(G) +w_{pm}(G)+\cdots + w_{p^km}(G).
\]

\end{thm}
\begin{proof} This follows from Theorem~\ref{thm:JL} and Proposition~\ref{prop:key}.
\end{proof}

 Let $G$ be a finitely generated pro-$p$-group. We write $\log P_{{\rm gr}(\F_p[[G]])}(t) = \sum_{n\geq 1} b_n t^n$ and recall that we have defined $w_n(G)$ by 
 \[
 w_n(G)=\frac{1}{n}\sum_{m\mid n} \mu(n/m) mb_m.
 \]
At  first glance the definition of $w_n$ may look a bit artificial. One may ask whether $w_n$ appears more naturally as the rank or dimension of some free finitely generated abelian group. Below we shall give a partial answer to this question.
Recall that for a profinite group $G$, the descending central series $(G_n)$ is defined inductively by
\[
G_1=G,\quad G_{n+1}=[G_n,G].
\]

Let $J(G)$ be the augmentation ideal of the completed group ring $\Z_p[[G]]$. (Here $\Z_p[[G]]$ and $J(G)$ are defined similarly to $\F_p[[G]]$ and $I(G)$.) 
Then we have two graded $\Z_p$-algebras associated to $G$ and $\Z_p[[G]]$ respectively which are defined by
\[
{\rm gr}_\gamma(G)= \bigoplus_{n\geq 1}  G_{n}/G_{n+1} \quad \text{ and } \quad  {\rm gr}(\Z_p[[G]]) =\bigoplus_{n\geq 0} J^n(G)/J^{n+1}(G).
\]
\begin{lem}
\label{lem:integral version}
Let $G$ be a finitely generated pro-$p$-group. Assume that the graded algebra ${\rm gr}_\gamma(G)=\bigoplus_{n\geq1} G_n/G_{n+1}$ is torsion free. Let $e_n(G)={\rank}_{\Z_p} G_n/G_{n+1}$.
\begin{enumerate}
 \item[(a)] The graded algebra ${\rm gr}(\Z_p[[G]])$ is a universal enveloping algebra of ${\rm gr}_\gamma(G)$.
\item[(b)] $J^n(G)/J^{n+1}(G)$ is a free module over $\Z_p$ of finite rank $d_n(G)$,  and   
\[
 P_{{\rm gr}(\Z_p[[G]])}(t) = 
 \sum_{n=0}^\infty d_n(G) t^n=\prod_{n=1}^\infty \frac{1}{(1-t^n)^{e_n(G)}}.
\]
\end{enumerate}
\end{lem}
\begin{proof} (a) This follows from \cite[Theorem 1.3]{Har}. In \cite[Theorem 1.3]{Har},  a discrete group $G$ is considered, but an adaptation of this proof to the profinite case is straightforward.

 (b) This follows from (a) and \cite[Proposition 2.5]{La3}.
\end{proof}

\begin{prop}
\label{prop:wn}
Let $G$ be a finitely generated pro-$p$-group. Assume that the graded algebra ${\rm gr}_\gamma(G)=\bigoplus_{n\geq1} G_n/G_{n+1}$ is torsion free. The following are equivalent.
\begin{enumerate}
 \item[(a)] ${\rm rank}_{\Z_p} J^n(G)/J^{n+1}(G)=\dim_{\F_p} I^n(G)/I^{n+1}(G)$ for all $n\geq 1$.
\item[(b)] $w_n(G)={\rm rank}_{\Z_p} G_n/G_{n+1}$ for all $n\geq 1$. 
\end{enumerate}
\end{prop}
\begin{proof} We keep the existing notation as in Lemma~\ref{lem:integral version}.

 (a) $\Rightarrow$ (b): Assume that  ${\rm rank}_{\Z_p} J^n(G)/J^{n+1}(G)=\dim_{\F_p} I^n(G)/I^{n+1}(G)$ for all $n$. Then by Theorem~\ref{thm:JL}, Remark~\ref{rmk:wn} and Lemma~\ref{lem:integral version}, we have
\[
 P_{{\rm gr}(\F_p[[G]])}(t)=\prod_{n=1}^\infty \frac{1}{(1-t^n)^{w_n(G)}}= P_{{\rm gr}(\Z_p[[G]])}(t)=\prod_{n=1}^\infty \frac{1}{(1-t^n)^{e_n(G)}}.
\]
Therefore $w_n(G)=e_n(G)$ for all $n\geq 1$. 

(b) $\Rightarrow$ (a): Assume that $w_n(G)=e_n(G)$ for all $n\geq 1$. Then by Theorem~\ref{thm:JL}, Remark~\ref{rmk:wn} and Lemma~\ref{lem:integral version}, we have
\[
   P_{{\rm gr}(\F_p[[G]])}(t)=P_{{\rm gr}(\Z_p[[G]])}(t).
\]
Therefore  ${\rm rank}_{\Z_p} J^n(G)/J^{n+1}(G)=\dim_{\F_p} I^n(G)/I^{n+1}(G)$ for all $n\geq 1$.
\end{proof}
\begin{rmk}
 We shall see in Sections 3 and 5, Remark~\ref{rmk:wn free} and Lemma~\ref{lem:wn Demushkin}, that both a free finitely generated pro-$p$-group and a Demushkin group with a relation of the form $r=[x_1,x_2]\cdots [x_{d-1},x_d]$ satisfy the equivalent statements in Proposition~\ref{prop:wn}. 
\end{rmk}

\begin{question}
 Let $G$ be a finitely generated pro-$p$-group. We assume that the graded algebra $\bigoplus_{n\geq1} G_n/G_{n+1}$ is torsion free. Is this true that
\[{\rm rank}_{\Z_p}(G_n/G_{n+1})=w_n(G)?\]
\end{question}

\section{Free pro-$p$-groups}
Throughout this section we assume that $S$ is a free pro-$p$-group on a finite set of  generators $x_1,\ldots,x_d$.
We recall the  Magnus homomorphism from the  completed group algebra $\F_p[[S]]$  to the $\F_p$-algebra $\F_p\langle\langle X_1,\ldots,X_d\rangle\rangle$ of the formal power series in $d$ non-commuting variables $X_1,\ldots,X_d$ over $\F_p$.
The homomorphism is given by 
\[
\psi\colon \F_p[[S]] \to \F_p \langle\langle X_1,\ldots,X_d\rangle\rangle, x_i\mapsto 1+X_i.
\]
The $\F_p$-algebra $\F_p\langle\langle X_1,\ldots,X_d\rangle\rangle$ is equipped with a natural valuation $v$ given by
\[
v(\sum a_{i_1,\ldots,i_k}X_{i_1}\cdots X_{i_k})=\inf\{k\mid a_{i_1,\ldots,i_k}\not=0\}\in \Z_{\geq 0}\cup \{\infty\},
\]
making it a compact topological $\F_p$-algebra. One basic result is the following.
\begin{lem} 
\label{lem:0a}
The Magnus homomorphism  $\psi$ is a (topological) isomorphism.
\end{lem}
\begin{proof} See for example, \cite[Chapter I, Proposition 7]{Se} or \cite[Chapter 6]{Laz}.
\end{proof}
\begin{cor} The Hilbert-Poincar\'e series
\[
P_{{\rm gr}(\F_p[[S]])}(t)=\frac{1}{1-dt}.
\]
 \end{cor}
\begin{proof}
 Via the Magnus homomorphism, the augmentation ideal $I(S)$ is mapped to the ideal $I=(X_1,\ldots,X_d)$ of $\F_p\langle\langle X_1,\ldots,X_d\rangle\rangle$. Hence 
\[
 a_n(S):= \dim_{\F_p}(I^n(S)/I^{n+1}(S))=\dim_{\F_p}(I^n/I^{n+1}),
\]
which is equal to the number of non-commutative monomials of degree $n$ in $d$ variables $X_1,\ldots,X_n$. Hence $a_n(S)=d^n$. The statement then follows. 
\end{proof}

We define $w_n(S)$ by 
\[w_n(S)=\frac{1}{n}\sum_{m\mid n} \mu(m) d^{n/m}.
\]
\begin{rmk}
\label{rmk:wn free}
Let $(S_n)$ be the lower central series of $S$. 
Then by Witt's result, $S_{n}/S_{n+1}$ is a free $\Z_p$-module of finite rank $w_n(S)$.
\end{rmk}
Theorem~\ref{thm:general}  immediately implies the following result.
\begin{prop}
\label{prop:cn free}
 If $n=p^k m $ with $(m,p)=1$, then 
\[
c_n(S) =w_m(S) +w_{pm}(S)+\cdots + w_{p^km}(S).
\qedhere
\]
\end{prop}
\begin{rmks}
\label{rmks:characterization}
 (1) If a finitely generated pro-$p$-group $G$ has Hilbert-Poincar\'e series of a  finitely generated free pro-$p$-group, then $G$ is itself free. In other words, if 
$ P_{{\rm gr}(\F_p[[G]])}(t)=\displaystyle \frac{1}{1-dt}, $
then $G$ is free of rank $d$. Indeed, we first have $c_1(G)=w_1(G)=d$, which is equal to the minimal number of topological generators of $G$. Hence there exists a minimal presentation of $G$:
\[
1 \to R\to S\to G\to 1,
\]
where $S$ is a free pro-$p$-group of rank $d$. We then have $c_n(G)=c_n(S)$ for all $n\geq 1$. Hence $|S/S_{(n)}|=|G/G_{(n)}|$ for all $n\geq 1$. Thus the natural epimorphism 
\[
S/S_{(n)} \surj G/G_{(n)}
\]
is in fact an isomorphism. This implies that $R\subseteq S_{(n)}$ for all $n\geq 1$. Therefore by \cite[Theorem 7.11]{Ko}, $R=1$ and hence $S\simeq G$.

(2) Let $G$ be a finitely generated pro-$p$-group. In the case  in which $G$ is realizable as  the Galois group of a maximal $p$-extension of a field $F$ containing a primitive $p$-th root of unity, it is noteworthy to point out that if $c_1(G)=c_1(S)$ and $c_2(G)=c_2(S)$ for some finitely generated free pro-$p$-group $S$, then $G$ is in fact isomorphic to $S$. Indeed, as $c_1(S)=c_1(G)$ we have a short exact sequence
\[
1 \to R\to S\stackrel{\pi}{\to} G\to 1.
\]
Since $c_1(S)=c_1(G)$ and $c_2(S)=c_2(G)$, we see that $|S/S_{(3)}|=|G/G_{(3)}|$. Thus the natural epimorphism 
\[
S/S_{(3)} \surj G/G_{(3)}
\]
is in fact an isomorphism. Hence by \cite[Theorem C]{EM} (see also \cite[Theorem D]{CEM} for the case $p=2$) we see that $\pi\colon S\to G$ is an isomorphism.
\end{rmks}
In Corollary~\ref{cor:unipotent} relying on Lemma~\ref{lem:coefficient}  below, we obtain an interesting purely group-theoretical corollary of our formula for $c_n(S)$. For each positive integer $n$, let $\U_{n+1}(\F_p)$ be the group of all upper-triangular unipotent $(n+1)\times(n+1)$-matrices with entries in  $\F_p$.
\begin{lem}
\label{lem:coefficient}
Let $n$ be a positive integer. If $\U_{n+1}(\F_p)_{(n)}=1$, then $c_n(S)=0$ for every free pro-$p$-group $S$.
\end{lem}
\begin{proof} Let $S$ be a free pro-$p$-group.
Assume that $\U_{n+1}(\F_p)_{(n)}=1$.  Then for  any (continuous) representation $\rho: S\to \U_{n+1}(\F_p)$, we have $\rho(S_{(n)})\subseteq \U_{n+1}(\F_p)_{(n)}=1$. Hence 
\[
S_{(n+1)}\subseteq S_{(n)}\subseteq \bigcap \ker(\rho\colon S\to \U_{n+1}(\F_p)),
\]
where $\rho$ runs over the set of all representations (continuous homomorphisms) $S \to \U_{n+1}(\F_p)$. On the other hand, we know that the Kernel Unipotent Conjecture is true for $S$ (see \cite{Ef1}, and  also \cite{Ef2}, \cite{MT}). This means that we have
\[
S_{(n+1)}=\bigcap \ker(\rho\colon S\to \U_{n+1}(\F_p)).
\]
Therefore, $S_{(n+1)}=S_{(n)}$, i.e., $c_n(S)=0$, as desired.
\end{proof}
\begin{cor}
\label{cor:unipotent}
Let $n$ be a positive integer. Then 
$ \U_{n+1}(\F_p)_{(n)}\simeq \F_p$ and
\[
n=\max\{h\mid \U_{n+1}(\F_p)_{(h)}\not=1\}.
\]
\end{cor}

\begin{proof} We first observe that if $S$ is a free pro-$p$-group of rank $d>1$, then all numbers $w_n(S)$, $n=1,2,\ldots$, are positive.  Therefore from Proposition~\ref{prop:cn free} we see that $c_n(S)\not=0$ for all $n\in \N$. Hence by Lemma~\ref{lem:coefficient}, $\U_{(n+1)}(\F_p)_{(n)}\not=1$.

On the other hand, it is well-known that $\U_{n+1}(\F_p)_{(n+1)}=1$. Hence $\U_{n+1}(\F_p)_{(n)}\subseteq Z(\U_{n+1}(\F_p))\simeq \F_p$, where  $Z(\U_{n+1}(\F_p))$ is the center of $\U_{n+1}(\F_p)$. Therefore 
\[
\U_{n+1}(\F_p)_{(n)}=Z(\U_{n+1}(\F_p))\simeq \F_p,
\]
and the second assertion is also clear.
\end{proof}
\begin{ex} Let $S$ be a free pro-$p$-group of finite rank $d$. We have
\[
\begin{aligned}
 c_1(S)&=d,\\
c_2(S)&= \begin{cases}
\frac{d^2-d}{2} \text{ if } p\not=2,\\
\frac{d^2+d}{2} \text{ if } p=2, 
\end{cases}\\
c_3(S) &=\begin{cases}
\frac{d^3-d}{3} \text{ if } p\not=3,\\
\frac{d^3+2d}{3} \text{ if } p=3,
\end{cases}\\
c_4(S)&= \begin{cases}
\frac{d^4-d^2}{4} \text{ if } p\not=2,\\
\frac{d^4+d^2+2d}{4} \text{ if } p=2, 
\end{cases}\\
c_5(S)&=\begin{cases}
\frac{d^5-d}{5} \text{ if } p\not=5,\\
\frac{d^5+4d}{5} \text{ if } p=5. 
\end{cases}
\end{aligned}
\]
Observe that our numbers $c_n(S)$, $n=1,2,\ldots$, also detect the minimal numbers  of generators of $S_{(n)}$. Indeed by the pro-$p$-version of Schreier's formula for each open subgroup $T$ of $S$, we have the following expression for the minimal number of generators $d(T)$ of $T$:
\[
 d(T)=[S:T](d(S)-1)+1.
\]
Therefore
\[
 d_n(S):=d(S_{(n)})= p^{\sum_{i=1}^{n-1}c_i(S)}(d-1)+1. 
\]
\end{ex}

Next we are going to give an explicit $\F_p$-basis for $S_{(n)}/S_{(n+1)}$, for each $n$. We shall first recall a definition of Hall commutators of weight $n$ and their linear order. This was originally introduced by M. Hall in \cite[Section 4]{Ha} (see also \cite[Definition 2.3]{Vo}).

\begin{defn} The set $C_n$ of {\it Hall commutators of weight $n$} together with a total order $<$ is inductively defined as follows:
\begin{enumerate}
 \item $C_1=\{x_1,\ldots,x_d\}$ with the ordering $x_1>\cdots >x_d$.
\item Assume $n>1$ and that we have defined Hall commutators and their ordering for all weights $<n$. 
Then $C_n$ is the set of all commutators $[c_1,c_2]$ where $c_1\in C_{n_1},c_2\in C_{n_2}$ such that $n_1+n_2=n$, $c_1>c_2$ and if $c_1=[c_3,c_4]$ then we also require that $c_2\geq c_4$. The set $C_n$ is ordered lexicographically, i.e., $[c_1,c_2]<[c_1^\prime,c_2^\prime]$ if and only if $c_1<c_1^\prime$, or  $c_1=c_1^\prime$ and $c_2<c_2^\prime$. Finally commutators of weight $n$ are greater than all commutators of smaller weight.
\end{enumerate}
\end{defn}
The following theorem was proved by M. Hall in the discrete case. The extension of his theorem to the pro-$p$ case is straightforward.
\begin{thm}[{\cite[Theorem 4.1]{Ha}}] The Hall commutators of weight $n$ represent a basis of $S_n/S_{n+1}$ as a free $\Z_p$-module. 

In particular $w_n(S)= |C_n|$.
 \end{thm}
The following theorem is due to Lazard (see \cite[Theorem 11.2]{DDMS}).
\begin{thm}[Lazard]
For each $n$, one has
\[
G_{(n)}=\prod_{ip^j\geq n}G_i^{p^j}.
\]
\end{thm}
\begin{cor}
Let us write $n=p^km$ with $(m,p)=1$. Then a basis of the $\F_p$-vector space $S_{(n)}/S_{(n+1)}$ can be represented by the following set
\[
C_m^{p^k} \bigsqcup C_{pm}^{p^{k-1}}\bigsqcup \cdots  \bigsqcup C_{p^{k-1}m}^p\bigsqcup C_n.
\]
\end{cor}
\begin{proof}
 By Lazard's theorem, we can check that the above set defines a set of generators for the $\F_p$-vector space $S_{(n)}/S_{(n+1)}$. Now by Proposition~\ref{prop:cn free} and by a counting argument, we see that this set defines a basis for the $\F_p$-vector space $S_{(n)}/S_{(n+1)}$. 
\end{proof}
\section{Free products of a finite number of cyclic groups of order 2}

\subsection{Free products of finitely many cyclic groups of order 2} Let $d$ be a non-negative integer. Let $G=C_2*\cdots * C_2$ be a free product in the category of pro-$2$-groups of  $d+1$ copies of $C_2$, where $C_2$ is the group of order 2. 

In this section we shall consider Pythagorean fields. A field $F$ is said to be Pythagorean if each finite sum of squares in $F$ is again a square in $F$. A Pythagorean field is called a formally real field if $-1$ is not a square. Pythagorean fields play an important role in Galois theory, real algebraic geometry and the algebraic theory of quadratic forms. We refer a reader to a beautiful exposition of related topics in \cite{Lam}. 

Let us recall that a formally real Pythagorean field $K$ with $|K^\times/(K^\times)^2|=2^{d+1}$ is called an SAP field if  $K$ admits exactly $d+1$ orderings. These SAP fields form an interesting and well investigated family of fields. (See \cite[Chapter 17]{Lam}.) 
\begin{thm}
\label{thm:SAP}
Let $F$ be a field with $|F^\times/(F^\times)^2|=2^{d+1}$. Then $F$ is an SAP field if and only if $G_F(2)$ is isomorphic to $G=C_2*\cdots *C_2$, the free product of $d+1$ copies of $C_2$.
\end{thm}
\begin{proof} 
The "only if" part follows from \cite{Mi}.

We now prove the "if" part. Suppose that $G_F(2)$ is isomorphic to $C_2*\cdots*C_2$ ($d+1$ copies of $C_2$). Then $F$ is formally real Pythagorean, and 
$|F^\times/(F^\times)^2|=2^{d+1}$. Now we pick any SAP field $K$ which has exactly $d+1$ orderings. From the "only if" part, we see that $G_F(2)\simeq G_K(2)$. In particular, $G_F(2)/G_F(2)_{(3)}\simeq G_K(2)/G_K(2)_{(3)}$. Then  \cite[Theorem 3.8]{MS2} implies that the Witt ring $WF$ of $F$ is isomorphic to the Witt ring $WK$ of $K$. We know that for a Pythagorean field $L$,  the Witt ring $WL$ of $L$ determines the space of  orderings $X_L$ of $L$. Hence the space of orderings of $F$ is isomorphic to the space of orderings of $K$. In particular $F$ admits exactly $d+1$ orderings. Therefore $F$ is an SAP field.
\end{proof}
\begin{cor}
\label{cor:SAP}
Let $F$ be any Pythagorean field, and let $K$ be an SAP field. Assume that  $|F^\times/(F^\times)^2|=|K^\times/(K^\times)^2|=2^{d+1}$ . Then there exists an epimorphism $G_K(2)\simeq C_2*\cdots*C_2 \twoheadrightarrow G_F(2)$.
\end{cor}
\begin{proof}
By \cite[Remark 17.5]{Lam}, $F$ has at least $d+1$ orderings, and we can choose $d+1$ involutions $\sigma_1,\ldots,\sigma_{d+1}$ in $G=G_F(2)$ such that $\sigma_1,\ldots,\sigma_{d+1}$ minimally generate $G_F(2)$.  The statement then follows from the previous theorem.

\end{proof}
Our treatment below is purely group-theoretical. However the group $G$ plays an important role as the maximal pro-$2$ quotient of the absolute Galois group of SAP fields. 
(We refer the interested reader to \cite{Haran}, \cite{Mi}.) Also it is interesting to observe that if $G=C_2*\cdots *C_2$ is the Galois group as above, then $G$ is already determined by its quotient $G/G_{(3)}$. (See also Remarks~\ref{rmks:SAP} for closely related observations.)
More precisely, assume that $H$ is another pro-2-group which is realizable as the Galois group of the maximal 2-extension of a field $F$, and that $H/H_{(3)}\simeq G/G_{(3)}$, then  $H\simeq G$. (See \cite{MS1,MS2,Mi}.)

The Hilbert-Poincar\'e series of ${\rm gr}(\F_2[[G]])$ is
\[
P_{{\rm gr}(\F_2[[G]])}(t)= \frac{1+t}{1-dt}.
\]
\begin{rmk} Since the cohomology algebra $H^*(C_2,\F_2)$ is isomorphic to  $\F_2[X]$, the polynomial algebra in one variable $X$ over $\F_2$, we have
\[
P_{H^*(G,\F_2)}(t)=\sum_{n=0}^\infty \dim_{\F_2}(H^n(G,\F_2)) t^n= 1+(d+1)t+(d+1)t^2+\cdots= \frac{1+dt}{1-t}
\]
by \cite[Theorems 4.1.4-4.1.5]{NSW}. Therefore, we have
\[
P_{H^*(G,\F_2)}(t) P_{{\rm gr}(\F_2[[G]])}(-t)=1.
\]
This is not a coincidence. One can show that the cohomology algebra $H^*(G,\F_2)$ is Koszul and that ${\rm gr}(\F_2[[G]])$ is its Koszul dual. The above equality is just a special case of the well-known relation between the two Hilbert-Poincar\'e series of a Koszul algebra and  its dual (see \cite[Corollary 2.2]{PP}).
\qed
\end{rmk}
We have
\[
\log P_{{\rm gr}(\F_2[[G]])}(t)= \log(\frac{1}{1-dt})-\log(\frac{1}{1+t})=\sum_{n\geq 1} \frac{1}{n} (d^n-(-1)^n)t^n.
\]
Now we define the sequence $w_n(G), n=1,2,\ldots$ by
\[
w_n(G)=\frac{1}{n}\sum_{m\mid n} \mu(n/m) (d^m-(-1)^m).
\]
\begin{prop}
\label{prop:cn free product of C2}
 If $n=2^k m $ with $(m,2)=1$, then 
\[
c_n(G) =w_m(G) +w_{2m}(G)+\cdots + w_{2^km}(G).
\qed
\]
\end{prop}

\subsection{Free products of the cyclic group of order 2 as  semidirect products}  Let $G=C_2*\cdots * C_2$ be a free product in the category of pro-$2$-groups of  $d+1$ copies of $C_2$ In this subsection we shall show that $G$ is isomorphic to a semidirect product $H\rtimes C_2$ of a free pro-$2$-group $H$ and $C_2$. We also provide provide a relation between $G_{(n)}$ and $H_{(n)}$.

We define the numbers $\epsilon_n$, $n=1,2,\ldots$ by
\[
\epsilon_n=\frac{1}{n}\sum_{m\mid n} \mu(n/m)(-1)^m,
\]
i.e., by  the M{\"o}bius inversion formula,
\[
\label{eq:epsilon}
\tag{*}(-1)^n= \sum_{m\mid n} m \epsilon_m.
\]
\begin{lem}
\label{lem:epsilon}
We have $\epsilon_1=-1$, $\epsilon_2=1$ and $\epsilon_n=0$ for $n\geq 3$.
\end{lem}
\begin{proof} The equation (\ref{eq:epsilon}) determines $\epsilon_n$, $n\in \N$, uniquely. But $\epsilon_1=-1$,  $\epsilon_2=1$ and $\epsilon_n=0$ for $n\geq 3$ work as for these numbers 
\[
\sum_{m\mid n} m \epsilon_m
=\begin{cases}
-1 &\;\text{if $n$ is odd},\\
-1+2=1 &\; \text{if $n$ is even}.
\qedhere
\end{cases}
\]

\end{proof}

Let us write 
\[G=C_2*C_2*\cdots *C_2=\langle x_0\mid x_0^2\rangle*\langle x_1\mid x_1^2\rangle *\cdots * \langle x_d\mid x_d^2\rangle.\] 
For  ease of notation, we consider $x_0,x_1,\ldots, x_d$ as elements of $G$. We consider a continuous homomorphism $\varphi:G\to C_2= \langle x\mid x^2\rangle$ defined by $x_i\mapsto x$ for all  $i=0,1,\ldots,d$. For each $i=1,\ldots,d$, we set $y_i=x_0x_i\in G$ and let $H$ be the closed subgroup of $G$ generated by $y_1,\ldots,y_d$.
\begin{lem} Let the notation be as above.
\begin{enumerate} 
\item $\ker\varphi=H$.
\item $H$ is a free pro-$2$-group of rank $d$.
\item We have $G\simeq H\rtimes C_2$, where the action of $C_2$ on $H$ is given by $xy_ix=y_i^{-1}$.
\end{enumerate}
\end{lem}
\begin{proof}
(1) Clearly $y_i\in \ker\varphi$, hence $H\subseteq \ker\varphi$. Now consider any element $\gamma\in \ker\varphi$. Then for each open neighborhood $U$ of $\gamma$ in $G$, there exists an element $g=x_{i_1}\cdots x_{i_r}\in U$, $i_1,\ldots,i_r\in \{1,\ldots,d\}$ such that $1=\varphi(g)=x^r$.   Hence $r=2s$ is even. Since $x_0 y_ix_0= y_i^{-1}$, we obtain
\[
g= x_0 y_{i_1}\cdots x_0 y_{i_r}=y_{i_1}^{-1}y_{i_2}\cdots y_{i_{r-1}}^{-1}y_{i_r}.
\]
Thus $g\in H$. Therefore $\gamma\in H$ and $H=\ker\varphi$.

(2) By profinite analogue of the well known Kurosch's subgroup theorem in the theory of free products of discrete groups due to E. Binz, J. Neukirch and G. Wenzel explained in \cite[Theorem 4.2.1 and Remarks below this Theorem]{NSW}, we see that $H$ is indeed a free pro-$2$-group of rank $d$.

(3) This follows by observing that $\psi\colon C_2=\langle x\mid x^2 \rangle \to G$ which maps $\bar{x}$ to $x_1$, is a section of $\varphi$. 
\end{proof}
The following proposition and corollary are remarkable properties of the pair $\{H,G\}$.  
\begin{prop}
\label{prop:comparison}
 We have $c_1(H)=d=c_1(G)-1$ and $c_n(H)=c_n(G)$ for all $n\geq 2$.
\end{prop}
\begin{proof} 
It is clear that $c_1(H)=w_1(H)=d$ and $c_1(G)=w_1(G)=d+1$. Hence $c_1(H)=d=c_1(G)-1$. We shall show that $c_n(H)=c_n(G)$ for any $n\geq 2$.

We note that 
\[ w_n(H)-w_n(G)=\frac{1}{n}\sum_{m\mid d}\mu(n/m)(-1)^m =\epsilon_n.
\] 
By Lemma~\ref{lem:epsilon}, one has $w_2(H)=w_2(G)+1$ and $w_n(H)=w_n(G)$  for every $n\geq 3$.

If $n>1$ is odd, then 
\[
c_n(H)=w_n(H)=w_n(G)=c_n(G).
\]
If $n$ is even, then by writing $n=2^km$ with $m$ odd, we have
\[
\begin{aligned}
c_n(H)&= w_m(H)+w_{2m}(H) +w_{4m}(H)+\cdots+w_{2^km}(H)\\
&=w_m(G)+w_{2m}(G) +w_{4m}(G)+\cdots+w_{2^km}(G)=c_n(G).
\end{aligned}
\]
(Note that we always have $w_m(H)+w_{2m}(H)=w_m(G)+w_{2m}(G)$ for every $m\geq 1$ odd.)
\end{proof}

\begin{cor}
\label{cor:quotient}
 Let $n\geq 2$ be an integer.
\begin{enumerate}
\item $H_{(n)}=H\cap G_{(n)}$.
\item $G/G_{(n)}\simeq H/H_{(n)}\rtimes C_2$, where the action of $C_2$ on $H/H_{(n)}$ is given by $\bar{x}\bar{y}_i\bar{x}=\bar{y}_i^{-1}$.
\end{enumerate}

\end{cor}
\begin{proof}
(1) Clearly $H_{(n)}\subseteq H\cap G_{(n)}$. We proceed by induction on $n$ to show that $H_{(n)}=H\cap G_{(n)}$. First consider the case $n=2$. We have an exact sequence
\[
1\to H/H\cap G_{(2)} \to G/G_{(2)}\to C_2\to 1.
\]
This implies that $[H: H\cap G_{(2)}]= [G:G_{(2)}]/2=2^d=[H:H_{(2)}]$. Hence $H_{(2)}=H\cap G_{(2)}$. Assume that $H_{(n)}=H\cap G_{(n)}$ for some $n\geq 2$. Then from the exact sequence
\[
1\to H/H\cap G_{(n)} \to G/G_{(n)}\to C_2\to 1,
\]
we obtain $[H:H_{(n)}]=[H:H\cap G_{(n)}]=[G:G_{(n)}]/2$. From a similar exact sequence we obtain
\[
\begin{aligned}
{[H:H\cap G_{(n+1)}]}&=\frac{1}{2}[G:G_{(n+1)}] = \frac{1}{2}[G:G_{n}] [G_{(n)}:G_{(n+1)}]\\
&= [H:H_{(n)}] [H_{(n)}:H_{(n+1)}]=[H:H_{(n+1)}].
\end{aligned}
\]
Here the equality $[G_{(n)}:G_{(n+1)}]=[H_{(n)}:H_{(n+1)}]$ follows from Proposition~\ref{prop:comparison}.
Therefore $H_{(n+1)}=H\cap G_{(n+1)}$.

(2) This follows from (1).
\end{proof}
\subsection{Another semidirect product}
In this subsection  we consider an example in which $G$ is the semidirect product  $G:=\Z_2^d \rtimes C_2= H \rtimes \langle x \rangle$, where the action of $C_2$ on $H:=\Z_2^d$ is given by $xyx=y^{-1}$, for all $y\in H$.  This group $G$ is  realizable as the maximal pro-$2$-quotient of the absolute Galois group of a superpythagorean field. 
Recall that a formally real Pythagorean field $F$ with $|F^\times/(F^\times)^2|=2^{d+1}<\infty$ is called a superpythagorean field if $F$ admits exactly $2^d$ orderings.
\begin{prop}
Let $F$ be a Pythagorean field with  $|F^\times/(F^\times)^2|=2^{d+1}$. Then there exists an epimorphism $G_F(2)\twoheadrightarrow G=\Z_2^d\rtimes C_2$.
\end{prop}
\begin{proof}
We choose any ordering $P$ in $F$ and an $\F_2$-basis $[a_1],\ldots, [a_d]$ of $P/(F^\times)^2$. By \cite{Be1} we know that there exists a field $E$, the Euclidean closure of $F$ with respect to $P$ such that $F(2)=E(\sqrt{-1})$, $E$ is a formally real field and $(E^\times)^2\cap F^\times=P$. We can pick for each $a_i$ as above, a sequence 
\[
\sqrt{a_i}, \sqrt[4]{a_i},\ldots, \sqrt[2^n]{a_i},\ldots,
\]
such that all $\sqrt[2^n]{a_i}$ are in $E^\times$. Indeed, by induction on $n$ we may assume that $\sqrt[2^n]{a_i}$ is in $E^\times$. Then we can pick $\sqrt[2^{n+1}]{a_i}$ in $(E^\times)^2$ because $E^\times=(E^\times)^2\cup -(E^\times)^2$. We set 
\[
\tilde{M}:=\bigcup_{n=1}^\infty F(\sqrt[2^n]{a_1},\ldots,\sqrt[2^n]{a_d}).
\]
Then $\tilde{M}$ is formally real since $\tilde{M}$ is a subfield of $E$. We recall that for each $n\in \N$,  $F(\sqrt{-1})$ contains a primitive $2^n$-th root of unity $\zeta_{2^n}$. (See \cite[Chapter II, Theorem 8]{Be}.)   We may also assume that $\zeta_{2^{n+1}}^2=\zeta_{2^n}$.
We set $M:=\tilde{M}(\sqrt{-1})$. Then $M/F$ is a Galois extension.

  We show that ${\rm Gal}(M/F{\sqrt{-1}})$ is isomorphic to $\Z_2^d$.  This follows from Kummer theory. In fact, let $\tau_1, \ldots, \tau_d$ be elements in ${\rm Gal}(M/F(\sqrt{-1})$ such that  for each $i=1,\ldots,d$, one has
\[
\tau_i(\sqrt[2^n]{a_i})=\zeta_{2^n} \sqrt[2^n]{a_i}, \; \text { and }\tau_i(\sqrt[2^n]{a_j})=\sqrt[2^n]{a_j}\; \forall j\not=i.
\]
Then ${\rm Gal}(M/F(\sqrt{-1}))=(\prod_{i=1}^d\langle \tau_i\rangle) \simeq \Z_2^d$. 

We observe that the restriction of a nontrivial element of ${\rm Gal}(E(\sqrt{-1})/E)$ to $M$ gives us a nontrivial element $\sigma\in {\rm Gal}(M/\tilde{M})$. Thus we have a spliting 
\[
{\rm Gal}(M/F)\simeq {\rm Gal}(M/F{\sqrt{-1}})\rtimes \langle \sigma \rangle,
\]
where $ \langle \sigma \rangle\simeq C_2$, and the action of $C_2$ on  ${\rm Gal}(M/F{\sqrt{-1}})$ is by involution.

The natural projection 
\[
G_F(2)={\rm Gal}(F(2)/F)\to {\rm Gal}(M/F)\simeq \Z_2^d\rtimes C_2
\]
gives the desired epimorphism.
\end{proof}

\begin{cor}
\label{cor:superPytha}
Let $F$ be a a field with $|F^\times/(F^\times)^2|=2^{d+1}$. Then $F$ is a superpythagorean field if and only if $G_F(2)$ is isomorphic to the group $G=\Z_2^d \rtimes C_2$.
\end{cor}
\begin{proof}
 Assume that $F$ is a superpythagorean field with $|F^\times/(F^\times)^2|=2^{d+1}$. Let the notation be as in the previous proposition. Then ${\rm Gal}(M/F)\simeq G=\Z_2^d \rtimes C_2$. On the other hand, from \cite[Example 3.8, (ii)]{Wa} (see also \cite[Chapter III, Theorem 1]{Be}), we know that ${\rm Gal}(M/F)$ is equal to $G_F(2)$. Hence $G_F(2)\simeq \Z_2^d \rtimes C_2$.

The converse direction is proved in a similar  fashion to the proof of the "if" part in Theorem~\ref{thm:SAP}.
\end{proof}
\begin{cor}
\label{cor:superPy}
Let $F$ be any Pythagorean field, and let $K$ be a superpythagorean field. Assume that  $|F^\times/(F^\times)^2|=|K^\times/(K^\times)^2|=2^{d+1}$ . Then there exists an epimorphism $G_F(2) \twoheadrightarrow G_K(2)\simeq \Z_2^d\rtimes C_2$.
\end{cor}
\begin{proof}
This follows from the previous theorem and corollary.
\end{proof}

\begin{lem} Let $G=H\rtimes \langle x \rangle=\Z_2^d\rtimes C_2$ be as above. Let $n\geq 2$ be an integer, and let $s=\lceil \log_2n \rceil$. Then $G_{(n)}=H^{2^s}$.
\end{lem}
\begin{proof} We proceed by induction on $n$. We first observe that $[y,x]=y^{-1}x^{-1} y x=(y^{-1})^2$ and $(yx)^2=y^2$, for every $y\in H$. Hence 
\[ 
G_{(2)}=G^2[G,G]=G^2=H^2.
\]
The lemma is true for $n=2$. We assume that the lemma is true for $j$ with  $2\leq j<n$. Then
\[
\begin{aligned}
 G_{(n)}&=G_{(\lceil n/2 \rceil)}^2\prod_{i+j=n}[G_{(i)},G_{(j)}]\\
&=G_{(\lceil n/2\rceil)}^2[G,G_{(n-1)}]\\
&=(H^{2^{s-1}})^2=H^{2^s}.
 \end{aligned}
\]
Here we use that $G_{n-1}\subseteq H^{2^{s-1}}$,  and hence $[G,G_{(n-1)}]\subseteq H^{2^s}$.
\end{proof}
An immediate  consequence of the above lemma is the following result.
\begin{cor}
\label{cor:cn superPy}
 Let $n\geq 2$ be an integer. We have
\[
c_n(G)=
\begin{cases} 
d+1 &\text{ if $n=1$},\\
d &\text{ if $n=2^s$ for some $1\leq s\in \Z$},\\
1 & \text{ if $n$ is not a power of 2}.
\end{cases}
\]
\end{cor}
\begin{cor}
\label{cor:superPy}
 We have
\[
 P_{{\rm gr}(\F_2[[G]])}(t)=\frac{1+t}{(1-t)^d}\prod_{i=1}^{\infty}\frac{1}{1-t^{2i+1}}. 
 \]
\end{cor}
\begin{proof}
We write $\log  P_{{\rm gr}(\F_2[[G]])}= \sum_{n\geq 1} b_n(G) t^n $, and let 
\[
w_n(G)=\frac{1}{n}\sum_{m\mid n} \mu(n/m) mb_m(G).
\]
By Lemma~\ref{lem:coprime}, if $n$ is odd then $w_n(G)=c_n(G)$. In particular, $w_1(G)=c_1(G)=d+1$, and $w_{2i+1}(G)=c_{2i+1}(G)=1$ for $i\geq 1$. 

By Lemma~\ref{lem:not coprime}, $w_2(G)=c_2(G)-c_1(G)=d-(d+1)=-1$.

We claim that $w_n(G)=0$ if $ n$ is even and $n\geq 4$. Indeed, if $n=2^s$ with $s\geq 2$, then by Lemma~\ref{lem:not coprime}, 
\[w_{2^s}(G)=c_{2^s}(G)-c_{2^{s-1}}(G)=d-d=0.\] 
Now if $n=2m$, where $m$ is not a power of $2$, then also by Lemma~\ref{lem:not coprime}, 
\[w_{2m}(G)=c_{2m}(G)-c_{m}(G)=1-1=0.\]
The corollary then follows from Remark~\ref{rmk:wn}.
\end{proof}

\begin{rmks}
\label{rmks:SAP}
It is an interesting fact that $c_1(G)$ and $c_2(G)$ can be sufficient in determining $G$ itself within some  large families of pro-$p$-groups. We mentioned an example in Remarks~\ref{rmks:characterization} (2). Here are two other instances.

(1) Suppose that $K$ is an SAP field with $|K^\times/(K^\times)^2|=2^{d+1}$. Then $G_K(2)=C_2*\cdots *C_2$, the free product of $d+1$ copies of $C_2$. 
By Proposition~\ref{prop:cn free product of C2}, one has
\[
c_1(G_K(2))=d+1 \text { and } c_2(G_K(2))=\frac{d(d+1)}{2}.
\]

Now let $F$ be a formally real Pythagorean field $F$ with $|F^\times/(F^\times)^2|< \infty$. We assume that $c_1(G_F(2))=d+1$ and that $c_2(G_F(2))=d(d+1)/2$ for some integer $d\geq 0$. Then we claim that $F$ is an SAP field with exactly $d+1$ orderings. So, quite remarkably, within the family of formally real Pythagorean fields with finitely many square classes, the numbers $c_1(G_F(2))$ and $c_2(G_F(2))$ above suffice to characterize SAP fields $F$. 
We shall now prove this claim. Because $c_1(G_F(2))=d+1$, we see that $G_F(2)$ has $d+1$ minimal generators, and therefore $|F^\times/(F^\times)^2|=2^{d+1}$. 
We pick any SAP field $K$ with $|K^\times/(K^\times)^2|=2^{d+1}$. 
 By Corollary~\ref{cor:SAP}, there exists an epimorphism $\varphi\colon G_K(2)\twoheadrightarrow G_F(2)$. 
 We have
\[
\begin{aligned}
|G_K(2)/{G_K(2)}_{(3)}| &=c_1(G_K(2))+c_2(G_K(2))\\
&=d+\frac{d(d+1)}{2}\\
&=c_1(G_F(2))+c_2(G_F(2))\\
&=|G_F(2)/{G_F(2)}_{(3)}|.
\end{aligned}
\]
This implies that the induced epimorphism $G_K(2)/{G_K(2)}_{(3)}\twoheadrightarrow G_F(2)/{G_F(2)}_{(3)}$ is an isomorphism. By \cite[Theorem D]{CEM}, $\varphi \colon G_K(2)\to G_F(2)$ is an isomorphism. This implies that $F$ is a SAP field by Theorem~\ref{thm:SAP}.

(2) Suppose that  $K$ is a superpythagorean field with $|K^\times/(K^\times)^2|=2^{d+1}<\infty$. By Corollary~\ref{cor:cn superPy}, one has
\[
c_1(G_K(2))=d+1 \; \text { and } c_2(G_K(2))=d.
\]

Now let $F$ be a formally real Pythagorean field $F$ with $|F^\times/(F^\times)^2|< \infty$. We assume that $c_1(G_F(2))=d+1$, $c_2(G_F(2))=d$ for some integer $d\geq 0$. Then we claim that $F$ is a superpythagorean field. So  within the family of formally real Pythagorean fields with finitely many square classes, the numbers $c_1(G_F(2))$ and $c_2(G_F(2))$ above, also suffice to characterize superpythagorean fields $F$. 
We shall now prove this claim. We pick any superpythagorean field $K$ with $|K^\times/(K^\times)^2|=2^{d+1}$. By Corollary~\ref{cor:superPy}, we have an epimorphism $\varphi\colon G_F(2)\twoheadrightarrow G_K(2)$. 
We have
\[
\begin{aligned}
|G_F(2)/{G_F(2)}_{(3)}| &=c_1(G_F(2))+c_2(G_F(2))\\
&=d+1+d\\
&=c_1(G_K(2))+c_2(G_K(2))\\
&=|G_K(2)/{G_K(2)}_{(3)}|.
\end{aligned}
\]
This implies that the induced epimorphism $G_F(2)/{G_F(2)}_{(3)}\twoheadrightarrow G_K(2)/{G_K(2)}_{(3)}$ is an isomorphism. By \cite[Theorem D]{CEM}, $\varphi \colon G_F(2)\to G_K(2)$ is an isomorphism. This implies that $F$ is a superpythagorean field by Corollary~\ref{cor:superPytha}.
\qed
\end{rmks}

\section{Demushkin groups}
Recall that a pro-$p$-group $G$ is said to be a Demushkin group if
\begin{enumerate}
\item $\dim_{\F_p} H^1(G,\F_p)<\infty,$ 
\item $\dim_{\F_p} H^2(G,\F_p)=1,$
\item  the cup product $H^1(G,\F_p)\times H^1(G,\F_p)\to H^2(G,\F_p)$ is a non-degenerate \mbox{bilinear} form.
\end{enumerate}
By the work of \cite{De1,De2}, \cite{Se1} and \cite{La1}, we now have a complete classification of Demushkin groups. 

Let $G$ be a Demushkin group of rank $d=\dim_{\F_p} H^1(G,\F_p)$. Let $c_n=c_n(G)$. Then by \cite[Theorem 5.1 (g)]{La3} (see also \cite{Fo,Ga,LM}), we have the Hilbert-Poincar\'e series 
\[
P_{{\rm gr}(\F_p[[G]])}(t)= \frac{1}{1-dt+t^2}.
\]
We write $1-dt+t^2=(1-at)(1-bt)$ so that $a+b=d$ and $ab=1$. Then
\[
\log P_{{\rm gr}(\F_p[[G]])}(t)= \log(\frac{1}{1-at})+\log(\frac{1}{1-bt})=\sum_{n\geq 1} \frac{1}{n}(a^n+b^n).
\]
We define the sequence $w_n(G), n=1,2,\ldots$ by
\[
w_n(G)=\frac{1}{n}\sum_{m\mid n} \mu(m) (a^{n/m}+b^{n/m})=\frac{1}{n}\sum_{m\mid n} \mu(n/m) (a^m+b^m).
\]
\begin{rmk} The numbers $w_n(G)$ are given by the formula
\[
w_n(G)=\frac{1}{n}\sum_{m\mid n} \mu(n/m) \left[ \sum_{0\leq i\leq [m/2]} (-1)^i \frac{m}{m-i} {m-i \choose i} d^{m-2i} \right].
\]
(See \cite[Proof of Proposition 4]{La2}.)
 \end{rmk}

\begin{prop} If $n=p^k m $ with $(m,p)=1$, then 
\[
c_n(G) =w_m(G) +w_{pm}(G)+\cdots + w_{p^km}(G).
\qedhere
\]
\end{prop}
\begin{ex} Let $G$ be a Demushkin pro-$p$-group of finite rank $d$. We have
\[
\begin{aligned}
 c_1(G)&=d,\\
c_2(G)&= \begin{cases}
\frac{d^2-d-2}{2} \text{ if } p\not=2,\\
\frac{d^2+d-2}{2} \text{ if } p=2, 
\end{cases}\\
c_3(G) &=\begin{cases}
\frac{d^3-4d}{3} \text{ if } p\not=3,\\
\frac{d^3-d}{3} \text{ if } p=3,
\end{cases}\\
c_4(G)&= \begin{cases}
\frac{d^4-5d^2+4}{4} \text{ if } p\not=2,\\
\frac{d^4-3d^2+2d}{4} \text{ if } p=2, 
\end{cases}\\
c_5(G)&=\begin{cases}
\frac{d^5-5d^3+4d}{5} \text{ if } p\not=5,\\
\frac{d^5-5d^3+9d}{5} \text{ if } p=5. 
\end{cases}
\end{aligned}
\]

Observe that our numbers $c_n(G)$, $n=1,2,\ldots$, also detect the minimal numbers  of generators of $G_{(n)}$. Indeed by the remarkable result of I. V. Ando\v{z}skii and independently by J. Dummit and J. Labute for each open subgroup $T$ of the Demushkin group $G$, we have the following expression for the minimal number of generators $d(T)$ of $T$:
\[
 d(T)=[G:T](d(G)-2)+2.
\]
(See \cite[Theorem 3.9.15]{NSW}.)
Therefore
\[
 d_n(G):=d(G_{(n)})= p^{\sum_{i=1}^{n-1}c_i(G)}(d-2)+2. 
\]
\end{ex}

From now on we assume that $G=F/\langle r\rangle$, where $F$ is a free pro-$p$-group on generators $x_1,x_2,\ldots,x_d$, and 
\[
 r=[x_1,x_2][x_3,x_4]\cdots [x_{d-1},x_d].
\]
Then we extract from \cite{La2} the following fact.
\begin{lem}
\label{lem:wn Demushkin} For every $n$, $w_n(G)=\rank_{\Z_p} G_n/G_{n+1}$.
\end{lem}
\begin{proof} This follows from  \cite[Theorem and proof of Proposition 4]{La2}. (Although \cite{La2} only treats abstract discrete groups, his results are also true for pro-$p$-groups with virtually the same proofs; one only has to replace $\Z$ by $\Z_p$, subgroups by closed subgroups, and group rings by completed group rings.)
\end{proof}

\begin{cor} Assume that for each $n$, $B_n$ represents a $\Z_p$-basis of $G_n/G_{n+1}$. 
Let us write $n=p^km$ with $(m,p)=1$. Then a basis of the $\F_p$-vector space $G_{(n)}/G_{(n+1)}$ can be represented by the following set
\[
B_m^{p^k} \bigsqcup B_{pm}^{p^{k-1}}\bigsqcup \cdots \bigsqcup B_{p^{k-1}m}^p\bigsqcup B_n.
\]
\end{cor}

\section{Some other groups}
\subsection{Free products of a finite number of Demushkin groups and free pro-$p$-groups}
Let $G$ be a free pro-$p$ product of $r$ Demushkin groups of ranks $d_1,\ldots, d_r$, and of a free pro-$p$-group of rank $e$. 
The Hilbert-Poincar\'e series of ${\rm gr}(\F_p[[G]])$ is
\[
P_{{\rm gr}(\F_p[[G]])}(t)=\frac{1}{1-(d_1+\cdots +d_r+e)t+ rt^2}=: \frac{1}{(1-at)(1-bt)}.
\]
We define the sequence $w_n(G), n=1,2,\ldots$ by
\[
w_n(G)=\frac{1}{n}\sum_{m\mid n} \mu(m) (a^{n/m}+b^{n/m})=\frac{1}{n}\sum_{m\mid n} \mu(n/m) (a^m+b^m).
\]
\begin{prop} If $n=p^k m $ with $(m,p)=1$, then 
\[
c_n(G) =w_m(G) +w_{pm}(G)+\cdots + w_{p^km}(G).
\qedhere
\]
\end{prop}
\subsection{A free product of a cyclic group of order 2  and a free pro-$2$-group}
We first consider the case of $p=2$  because this is  the case of interest in Galois theory of $2$-extensions, and because this case is a bit simpler than the general case of any prime $p$. This latter case will be covered in the next subsection.

Let $G=C_2* S$ be a free pro-$2$ product of the cyclic group $C_2$ of order 2 and a free  pro-$2$-group of rank $d$.
The Hilbert-Poincar\'e series of ${\rm gr}(\F_2[[G]])$ is
\[
P_{{\rm gr}(\F_2[[G]])}(t)=\left(\frac{1}{1+t}-dt\right)^{-1}=\frac{1+t}{1-dt- dt^2}=: \frac{1+t}{(1-at)(1-bt)}.
\]
We define the sequence $w_n(G), n=1,2,\ldots$ by
\[
w_n(G)=\frac{1}{n}\sum_{m\mid n} \mu(n/m) (a^m+b^m-(-1)^m).
\]
\begin{prop} If $n=2^k m $ with $(m,2)=1$, then 
\[
c_n(G) =w_m(G) +w_{pm}(G)+\cdots + w_{2^km}(G).
\qedhere
\]
\end{prop}
\subsection{A free product of a cyclic group of order $p$  and a free pro-$p$-group}
Let $G=C_p* S$ be a free pro-$p$ product of the cyclic group $C_p$ of order $p$, and a free  pro-$p$-group of rank $d$. We shall find a formula for $c_n(G)$. 
The Hilbert-Poincar\'e series of ${\rm gr}(\F_p[[G]])$ is
\[
P_{{\rm gr}(\F_p[[G]])}(t)=\frac{1+t+\cdots+t^{p-1}}{1-dt- dt^2-\cdots-dt^p}=: \frac{(1-\xi_1 t)\cdots(1-\xi_{p-1}t)}{(1-a_1t)\cdots (1-a_pt)}.
\]
We define the sequence $w_n(G), n=1,2,\ldots$ by
\[
w_n(G)=\frac{1}{n}\sum_{m\mid n} \mu(n/m)( a_1^m+\cdots+a_p^m-(\xi_1^m+\cdots+\xi_{p-1}^m)).
\]
\begin{prop} If $n=p^k m $ with $(m,p)=1$, then 
\[
c_n(G) =w_m(G) +w_{pm}(G)+\cdots + w_{p^km}(G).
\qedhere
\]
\end{prop}
\begin{rmk}
Note that
\[
\xi_1^n+\cdots+\xi_{p-1}^n
=\begin{cases}
-1 \text{ if } (n,p)=1,\\
p-1 \text{ if } p\mid n.
\end{cases}
\]
We shall compute $a_1^n+\cdots+a_{p}^n$.
From
\[
  \frac{1}{(1-a_1t)\cdots(1-a_pt)}=\frac{1}{1-(dt+dt^2+\cdots+dt^p)},
\]
taking the logarithms of both sides, we obtain
\[
\begin{aligned}
& \sum_{n\geq 1}\frac{1}{n}(a_1^n+\cdots+a_p^n)t^n= \sum_{n\geq 1} \frac{1}{n} (dt+dt^2+\cdots +dt^p)^n\\
 &= \sum_{n\geq 1} \frac{1}{n}\sum_{\substack{k_1+\cdots+k_p=n,\\k_i \geq 0}} \binom{n}{k_1,\ldots,k_p} (dt)^{k_1} (dt^2)^{k_2}\cdots (dt^p)^{k_p}\\
&= \sum_{M}\sum_{\substack{k_1+2k_2+\cdots+pk_p=M,\\k_i \geq 0}}\\
&\hspace*{60pt}
\left[\frac{1}{M-k_2-\cdots-(p-1)k_p} \binom{M-k_2-\cdots-(p-1)k_p}{k_1,\ldots,k_p} d^{M-k_2-\cdots-(p-1)k_p}\right] t^M.
\end{aligned}
\]
Finally comparing the coefficients of $t^n$ gives us the required formula for $a_1^n+\cdots+a_{p}^n$,
\[
\begin{aligned}
 &a_1^n+\cdots+a_p^n\\
&=\sum_{\substack{k_1+2k_2+\cdots+pk_p=n,\\k_i \geq 0}}\frac{n}{n-k_2-\cdots-(p-1)k_p} \binom{n-k_2-\cdots-(p-1)k_p}{k_1,\ldots,k_p} d^{n-k_2-\cdots-(p-1)k_p}.
 \end{aligned}
\]
\end{rmk}

\end{document}